\pdfoutput=1

\documentclass[11pt]{amsart}

\usepackage{mathtext}
\usepackage[T2A]{fontenc}
\usepackage[russian,english]{babel}
\usepackage[margin=1.1in]{geometry}
\usepackage{graphicx}
\usepackage{amsfonts,amssymb,amscd,amsmath,amsthm}
\usepackage{mathrsfs}
\usepackage{epsf}
\usepackage{wrapfig}
\usepackage{cite}
\usepackage[hyper]{amsbib}

\newcommand{\R}{\mathbb{R}}
\DeclareMathOperator \conv {conv}

\begin{document}

\theoremstyle{plain}
\newtheorem{theorem}{Theorem}
\newtheorem{lemma}{Lemma}
\newtheorem{conjecture}{Conjecture}

\title{1-skeletons of the spanning tree problems with~additional constraints}

\author{Vladimir Bondarenko, Andrei Nikolaev, Dzhambolet Shovgenov}
\thanks {The research was partially supported by the Russian Foundation for Basic Research, Project 14-01-00333, and the President of Russian Federation Grant MK-5400.2015.1}

\address{%
Department of Discrete Analysis, P.G. Demidov Yaroslavl State University, Sovetskaya, 14, Yaroslavl, 150000, Russia
}
\email {bond@bond.edu.yar.ru, andrei.v.nikolaev@gmail.com, djsh92@mail.ru}

\begin{abstract}
We consider the polyhedral properties of two spanning tree problems with additional constraints. In the first problem, it is required to find a tree with a minimum sum of edge weights among all spanning trees with the number of leaves less or equal a given value. In the second problem, an additional constraint is the assumption that the degree of all vertices of the spanning tree does not exceed a given value. The decision versions of both problems are NP-complete.

We consider the polytopes of these problems and their 1-skeletons. We prove that in both cases it is a NP-complete problem to determine whether the vertices of 1-skeleton are adjacent. Although it is possible to obtain a superpolynomial lower bounds on the clique numbers of these graphs. These values characterize the time complexity in a broad class of algorithms based on linear comparisons. The results indicate a fundamental difference in combinatorial and geometric properties between the considered problems and the classical minimum spanning tree problem.
\end {abstract}


\maketitle


\section {Introduction}

A significant number of works related to the computational complexity of combinatorial problems is aimed at the study of geometric objects associated with problems.
Usually such objects are the polytopes of the problems and their 1-skeletons.
In particular, the clique number of the 1-skeleton (the size of the maximum clique) of the problem serves as a lower bound on computational complexity in a broad class of algorithms based on linear comparisons.
Moreover, it was found out that this characteristic is polynomial for known polynomially solvable problems and superpolynomial for intractable problems (see, for example, \cite {Bondarenko-DAN,Bondarenko-Maximenko,Bondarenko-Nikolaev}).

It is well-known that polynomially solvable problems may become NP-complete with the introduction of additional constraints.
Sometimes the opposite happens: the problem is NP-complete, however, the introduction of additional constraints makes it possible to design an effective algorithm for it.
In this regard, the question arises: how does the introduction of additional constraints affect the characteristics of the 1-skeleton of the problem?

We consider the combinatorial optimization problems on graphs that admit the following formulation: let $G=(V,E)$ be an edge-weighted graph and $T$ be some set of its subgraphs, it is required to find a subgraph of $T$, having a minimum (or maximum) possible edge weight.

\textbf{Minimum spanning tree (MST).}
In this classical problem it is required to find a spanning tree with a minimum edge weight in a connected graph $G$.

The problem is polynomially solvable, for example, by the algorithms of Prim and Kruskal \cite {Papa-St}.

\textbf{Leaf-constrained minimum spanning tree (LCMST)}.
Given a connected graph $G(V,E)$ and a positive integer $k < |V|$, it is required to find a spanning tree with a minimum edge weight such that $k$ or less vertices have degree $1$.

\textbf{Restricted-leaf-in-subgraph minimum spanning tree (RLSMST)}.
Let $G(V,E)$ be a connected graph, $U$ a vertex subset of $G$, and a positive integer $k < |U|$, it is required to find a spanning tree of $G$ with a minimum edge weight such that the number of leaves belonging to $U$ is less than or equal to $k$.

\textbf {Set version of leaf-constrained minimum spanning tree (SVMST)}.
For a connected graph $G(V,E)$ and some subset $U$ of its vertices it is required to find a spanning tree of $G$ with a minimum edge weight such that all leaves belong to the set $U$.

\textbf{Degree-constrained minimum spanning tree (DCMST)}.
Given a connected graph $G(V,E)$ and a positive integer $k < |V|$, it is required to find a spanning tree with a minimum edge weight in which no vertex has a degree larger than $k$.

Unlike the simple spanning tree problem, for all the above problems already the decision versions (does there exist at least one spanning tree in the graph $G$ that satisfies the additional constraints) are NP-complete.

A significant number of works are devoted to approximation algorithms for the spanning tree problems with restrictions on the number of leaves and degree of vertices\cite {Rahman, Fernandes, Goemans}. In particular, a linear 2-approximation algorithm for the dual problem of constructing a spanning tree with a maximum number of inner vertices \cite {Salamon}, and a polynomial time algorithm that returns a spanning tree with degrees of vertices at most $k+1$, and the edge weight not exceeding the weight of the optimal spanning tree with degrees of vertices at most $k$ \cite {Singh}.

\section {Polytope of the problem}

We consider the general problem described above on the graph $G=(V,E)$ with the set $T$ of its subgraphs.
Let $|V|=n$, we denote by $d$ the number of edges of the complete graph:
$$d=|E|=\frac{n(n-1)}{2}.$$ 
We consider the space $\R^{d}$ where the coordinates are associated with the edges of $G$.
For each element $t$ of $T$, we construct its characteristic vector $x=x(t)\in \R^{d}$ where some coordinate is equal to 1 if the corresponding edge belongs to $t$, otherwise the coordinate is equal to $0$.
We denote the set of all characteristic vectors by $X$.
Let $c\in \R^{d}$ be the vector composed of the edge weights of the graph $G$, then the problem is to find the maximum of a linear function $\left\langle c,x\right\rangle$ over the finite set $X$.

Let $P(X)$ be the polytope of the problem: $P(X) = \conv X$.
The skeleton of some polytope $P$ (also called $1$-skeleton) is the graph whose vertex set is the vertex set of $P$ (in this case it is $X$) and edge set is the set of 1-faces of $P$.
To study the skeleton of the polytope, the following assertion is useful (see, for example, \cite {Bron}).

\begin {lemma} \label {lemma_adjacency_separation}
Two vertices of the polytope $P$ are adjacent if and only if they are strictly separated from the rest of the vertices of $P$.
In other words, vertices $x$ and $y$ of the polytope $P$ are nonadjacent if and only if some convex combination of $x$ and $y$ coincides with a convex combination of the rest of the vertices: there exist $\alpha \geq 0, \beta \geq 0, \gamma_{z} \geq 0$ for which
\begin {gather*}
\alpha x + \beta y = \sum {\gamma_{z} z}, \\
\alpha + \beta = \sum {\gamma_{z}} = 1,
\end {gather*}
and the sum is taken over all vertices other than $x$ and $y$.
\end {lemma}

\section {Spanning tree polytope}

A complete $H$-representation of the polytope $MST_{n}$ of the spanning tree problem in the graph $G(V,E)$ on $n$ vertices is known and has the form
\begin{gather}
\sum_{e\in E} {x_{e} = n - 1}, \label {first_MST_H}\\
\sum_{e\in E(S)} {x_{e} \leq |S|-1, \forall S\subset V},\\
x_{e} \geq 0, \forall e \in E. \label {last_MST_H}
\end{gather}

If we introduce some extra variables, then the system (\ref{first_MST_H})-(\ref{last_MST_H}) can be rewritten in an equivalent form with a polynomial ($O(n^{3})$) number of constraints, it is called an extended formulation of $MST_{n}$ \cite {Martin}. Thus, the problem can be solved in polynomial time by linear programming.

The skeleton of $MST_{n}$ is completely described, the exact clique number is given in \cite {Belov}.

\begin {theorem}
The clique number of $MST_{n}$ skeleton is polynomial in $n$:
$$\omega (MST_{n}) = \left\lfloor \frac {n^{2}}{4} \right\rfloor.$$
\end {theorem}

\section {Leaf-constrained minimum spanning tree}

In contrast to the general problem, a complete $H$-representation of the polytope of the leaf-constrained minimum spanning tree is not known.
Integer programming formulation of the problem is obtained by supplementing the system (\ref{first_MST_H})-(\ref{last_MST_H}) with constraints
\begin {gather}
\sum_{e \in \delta_{v}}{x_{e} + (|\delta_{v}|-1)y_{v}} \leq |\delta_{v}|,\forall v \in V,\\
x_{e}, y_{v} \in \{0,1\},\forall e \in E, v \in V, \label {last_LCMST_H}
\end {gather}
where $\delta_{v}$ is the set of edges incident to the vertex $v$, and the additional variables $y_{v}$ correspond to the leaves.

This formulation is usually used for the problem of optimizing the number of leaves
$$\sum_{v \in V} {y_{v}} \rightarrow \max (\min).$$
A variant of the problem with optimization of the weight of the spanning tree can be obtained by adding to the system (\ref{first_MST_H})-(\ref{last_LCMST_H}) the inequalities
$$\sum_{v \in V} {y_{v}} \leq k$$
for the problem with a simple restriction on the number of leaves ($LCMST_{n,k}$),
$$\sum_{v \in U} {y_{v}} \leq k$$
for the problem with a restriction in a subgraph ($RLSMST_{n,U,k}$), and
$$\forall v \in V \backslash U:  y_{v} = 0$$
for the set version of leaf-constrained minimum spanning tree ($SVMST_{n,U}$).

We consider the minimum spanning tree problem with a restriction on the number of leaves. 
Let $|V|=n$, and $k$ is the permitted number of terminal vertices. 
We construct a spanning tree $t$ of a special form.
We choose two vertices $u$, $w$ from $V$ and a set $V_{uw}$ consisting of $k$ vertices, where some of them $V_{u} =\{v_{1},\ldots, v_{s}\}$ are adjacent to $u$, while the others $\{v_{s+1},\ldots ,v_{k}\} = V_{w}$ are adjacent to $w$.
The remaining $n-k-2$ vertices are connected by edges only with each other, or with vertices $u$ and $w$ so that the result is a spanning tree (Fig. \ref {Fig_1}).

\begin{figure}[h]
	\centering
	\includegraphics[width=5in]{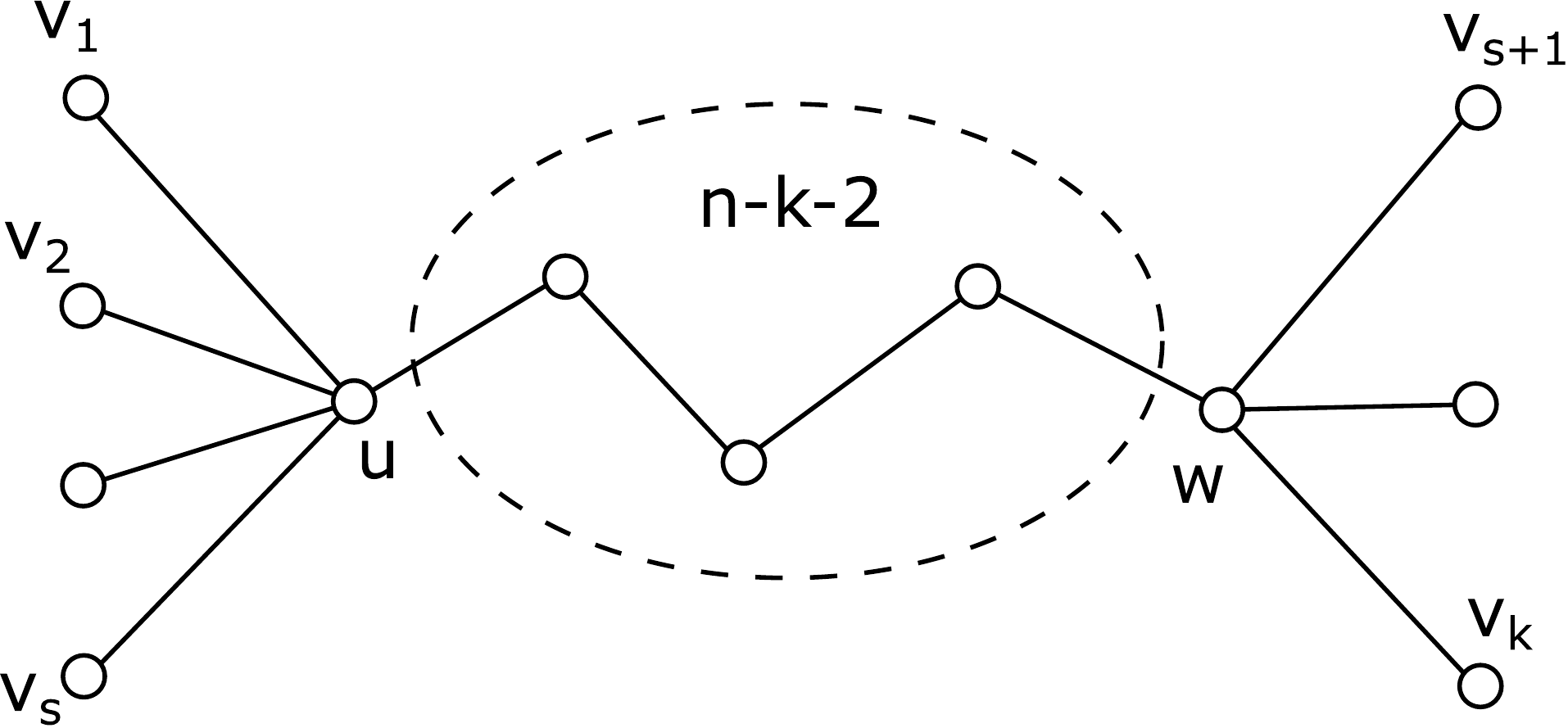}
	\caption {Spanning tree of a special form with $k$ leaves.}
	\label {Fig_1}
\end{figure}

\begin {lemma} \label {lemma_discarding_vertices_HP_LC}
The graph $t_{h}$, obtained from the tree $t$ by discarding the vertices $v_{1}, v_{2},\ldots, v_{k}$ with the corresponding edges $(v_{i},u)$ and $(v_{j},w)$, is a Hamiltonian path on $n-k$ remaining vertices.
\end {lemma}
\begin {proof}
The graph $t_{h}$ is a spanning tree on the vertex set $V\backslash V_{uw}$.
Therefore, it has at least two terminal vertices.
These vertices can only be $u$ and $w$, since if any other vertex from $V\backslash V_{uw}$ is a leaf in the tree $t_{h}$, then it is a leaf in the tree $t$ as well.
But the number of leaves in $t$ cannot exceed $k = |V_{uw}|$.
Therefore, $t_{h}$ is a spanning tree that has exactly two terminal vertices: $u$ and $w$.
Thus, $t_{h}$ is a simple path passing through all vertices of $V\backslash V_{uw}$ such that $u$ and $w$ are terminal vertices.
\end {proof}

We fix the sets $V_{u}$ and $V_{w}$ and consider the set $T_{k}$ of all spanning trees of the form described above with $k$ leaves.
By Lemma \ref {lemma_discarding_vertices_HP_LC}, every such tree contains a path $t_{h}$ with terminal vertices $u$ and $w$ passing through all vertices of $V\backslash V_{uw}$.
The converse is also true: for every such path, there corresponds a tree from $T_{k}$.
Let $HP_{uw}$ be the convex hull of the characteristic vectors of the Hamiltonian paths $t_{h}$ between the vertices $u$ and $w$.

\begin {lemma} \label {lemma_LCMST_adjacency}
Vertices $x$ and $y$ of the polytope $LCMST_{n,k}$ that correspond to the trees of $T_{k}$ are nonadjacent if and only if the corresponding vertices $x_{h}$ and $y_{h}$ of $HP_{uw}$ are nonadjacent as well.
\end {lemma}

\begin {proof}

We suppose that the vertices $x_{h}$ and $y_{h}$ of the polytope $HP_{uw}$ are nonadjacent.
By Lemma \ref {lemma_adjacency_separation}, there are nonnegative $\alpha, \beta, \gamma_{z}$ such that $\alpha + \beta = \sum {\gamma_{z}} = 1$ and the following holds:
\begin {equation}
\alpha x_{h} + \beta y_{h} = \sum {\gamma_{z} z_{h}}, \ z_{h}\in HP_{uw}. \label {nonadjacent_HP_uw}
\end {equation}
There is a bijection between the Hamiltonian paths of $HP_{uw}$ and the spanning trees of $T_{k}$.
If we supplement (\ref {nonadjacent_HP_uw}) with equalities for the coordinates that correspond to the edges $(v_{i},u)$ and $(v_{j},w)$, then we obtain
$$\alpha x + \beta y = \sum {\gamma_{z} z}, \ z \in T_{k},$$
hence the vertices $x$ and $y$ of the polytope $LCMST_{n,k}$ are nonadjacent.

Now we suppose that the vertices $x$ and $y$ are nonadjacent, then there are nonnegative $\alpha, \beta, \gamma_{z}$ such that $\alpha + \beta = \sum {\gamma_{z}} = 1$ and
$$\alpha x + \beta y = \sum {\gamma_{z} z}.$$

For points $x$ and $y$, all the coordinates that correspond to the edges incident to the vertices $v_{1}, v_{2},\ldots , v_{k}$ coincide, since these edges are fixed for spanning trees of $T_{k}$, and consequently these coordinates coincide also at the points $z$, therefore
$$\alpha x + \beta y = \sum {\gamma_{z} z}, \ z \in T_{k}.$$
There is a bijection between the spanning trees of $T_{k}$ and the Hamiltonian paths of $HP_{uw}$. Thus,
$$\alpha x_{h} + \beta y_{h} = \sum {\gamma_{z} z_{h}}, \ z_{h}\in HP_{uw},$$
and the vertices $x_{h}$ and $y_{h}$ of the polytope $HP_{uw}$ are nonadjacent.
\end {proof}

By Lemma \ref {lemma_LCMST_adjacency}, we can transfer the properties of the traveling salesman polytope to $LCMST_{n,k}$.
It is sufficient to consider the following simple fact: two vertices of the polytope $HP_{uw}$ of the Hamiltonian paths are adjacent if and only if the vertices of the traveling salesman polytope that correspond to the Hamiltonian cycles constructed by merging two terminal vertices into one are adjacent as well.
Thus, from Lemma \ref {lemma_LCMST_adjacency} and the well-known result of Papadimitriou \cite{Papa}, it follows

\begin {theorem}
The question whether two vertices of $LCMST_{n,k}$ are nonadjacent is NP-complete.
\end {theorem}

Despite the complexity of $LCMST_{n,k}$ skeleton, we can obtain a superpolynomial lower bound on its clique number.

\begin {theorem} \label {theorem_LC_clique}
The clique number of $LCMST_{n,k}$ skeleton is superpolynomial in $n$:
$$\omega(LCMST_{n,k}) \geq 2^{(\sqrt {\left\lfloor \frac {n-k-1}{2} \right\rfloor} - 9) \slash 2}.$$
\end {theorem}

For the proof of Theorem \ref {theorem_LC_clique} it suffices to apply Lemma \ref {lemma_LCMST_adjacency} and the lower bound on the clique number of the skeleton of the traveling salesman polytope $TSP_{n}$ \cite{Bondarenko-DAN,Bondarenko-Maximenko}:
$$\omega(TSP_{n}) \geq 2^{(\sqrt {\left\lfloor \frac {n}{2} \right\rfloor} - 9) \slash 2)}.$$

Restricted-leaf-in-subgraph and set version of leaf-constrained minimum spanning tree problems are considered similarly.
In the first case, for $RLSMST_{n,U,k}$, it suffices to take the subgraph on the vertices $U$ instead of the graph $G$ and construct the corresponding spanning tree.
In the second case, for $SVMST_{n,U}$, it suffices to identify the set of leaves $V_{uw}$ with the set $U$.

\begin {theorem}
The question whether two vertices of $RLSMST_{n,U,k}$ are nonadjacent is NP-complete.
\end {theorem}

\begin {theorem}
The clique number of $RLSMST_{n,U,k}$ skeleton is superpolynomial in $|U|$:
$$\omega(RLSMST_{n,U,k}) \geq 2^{(\sqrt {\left\lfloor \frac {|U|-k-1}{2} \right\rfloor} - 9) \slash 2}.$$
\end {theorem}

\begin {theorem}
The question whether two vertices of $SVMST_{n,U}$ are nonadjacent is NP-complete.
\end {theorem}

\begin {theorem}
The clique number of $SVMST_{n,U}$ skeleton is superpolynomial in $n$:
$$\omega(SVMST_{n,U}) \geq 2^{(\sqrt {\left\lfloor \frac {n-|U|-1}{2} \right\rfloor} - 9) \slash 2}.$$
\end {theorem}

\section {Degree-constrained minimum spanning tree}

Now we consider the problem of constructing a minimum spanning tree in which no vertex has a degree larger than $k$.
As for the problem with restriction on the number of leaves, a complete $H$-representation of the polytope $DCMST_{n,k}$ is not known \cite {Goemans}.
Integer programming formulation of the problem is obtained by supplementing the system (\ref{first_MST_H})-(\ref{last_MST_H}) with constraints
\begin {gather*}
\sum_{e \in \delta_{v}} {x_{e}} \leq k, \\
x_{e} \in {0,1}, v \in V.
\end {gather*}

For $n>2$ and $k>1$ we denote by
$$s=\left\lfloor \frac {n-2}{k-1} \right\rfloor$$
and construct a tree $t$ of a special form.
We divide the set of vertices into $s$ subsets of the form $V_{i}=\{v_{i},v_{i,1},\ldots, v_{i,k-2}\}$ with $k-1$ vertices.
All the remaining vertices, there are from $2$ to $k + 1$ of them, are divided into two subsets $V_{0} = \{v_{0},v_{0,1},...,v_{0,p}\}$ and $V_{s+1}=\{v_{s+1},v_{s+1,1}\ldots,v_{n}\}$.
We consider the tree of the following form: in each subset $V_{i}$ all vertices are adjacent only with the vertex $v_{i}$ (Fig. \ref {Fig_2}).
Note that the degrees of the vertices $v_{0}$ and $v_{s+1}$ cannot exceed $k$ by construction.

\begin{figure}[h]
	\centering
	\includegraphics[width=6in]{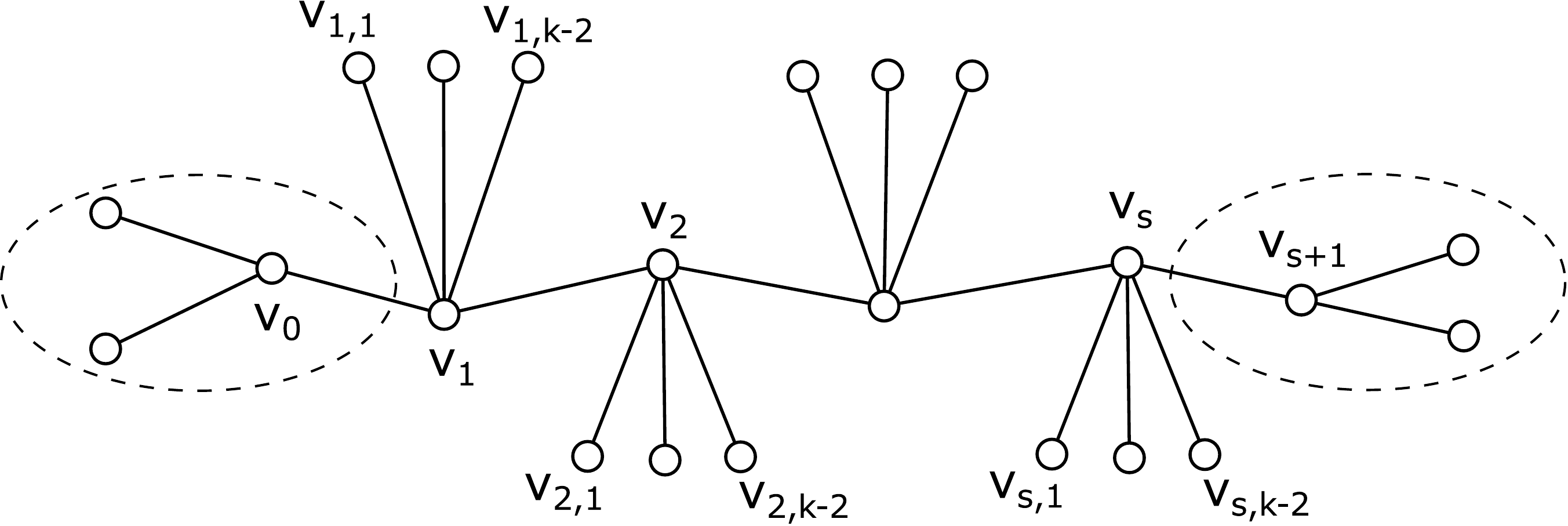}
	\caption {Spanning tree of a special form with vertex degree not exceeding $k$.}
	\label {Fig_2}
\end{figure}

\begin {lemma} \label {lemma_discarding_vertices_HP_DC}
The graph $t_{h}$, obtained from the tree $t$ by discarding the vertices $v_{0}$, $v_{s+1}$, and $v_{i,j}$ with the corresponding edges $(v_{i,j},v_{i})$,
is a Hamiltonian path with terminal vertices $v_{1}$ and $v_{s}$.
\end {lemma}

\begin {proof}
By construction, the degrees of the vertices $v_{1}$ and $v_{s}$ cannot be less than $k-1$, and the degrees of the vertices $\{v_{2},\ldots, v_{s-1}\}$ less than $k-2$.
Since the degree of each vertex in the tree $t$ cannot exceed $k$, the graph obtained after discarding corresponding vertices can only be a Hamiltonian path between $v_{1}$ and $v_{s}$.
\end {proof}

We consider the collection $T_{k}$ of all spanning trees of the type described above.
By Lemma \ref {lemma_discarding_vertices_HP_DC}, every tree of $T_{k}$ contains a path $t_{h}$ with terminal vertices $v_{1}$ and $v_{s}$ passing through vertices $\{v_{2},\ldots, v_{s-1}\}$.
The converse is also true: for every such path, there is a tree of $T_{k}$.
We denote by $HP_{1s}$ the convex hull of the characteristic vectors of the Hamiltonian paths between the vertices $v_{1}$ and $v_{s}$.

\begin {lemma}
Vertices $x$ and $y$ of the polytope $DCMST_{n,k}$ that correspond to the trees of $T_{k}$ are nonadjacent if and only if the corresponding vertices $x_{h}$ and $y_{h}$ of $HP_{1s}$ are nonadjacent as well.
\end {lemma}

The proof is similar to the proof of Lemma \ref {lemma_LCMST_adjacency}. As a corollary, we obtain the following assertions.

\begin {theorem}
The question whether two vertices of $DCMST_{n,k}$ are nonadjacent is NP-complete.
\end {theorem}

\begin {theorem}
The clique number of $DCMST_{n,k}$ skeleton is superpolynomial in $s$:
$$\omega(DCST_{n,k}) \geq 2^{(\sqrt {\left\lfloor \frac {s-1}{2} \right\rfloor} - 9) \slash 2}.$$
\end {theorem}

\section {Conclusion}

Thus, the general minimum spanning tree problem and the problems with additional constraints on the number of leaves and the degree of vertices have fundamentally different polyhedral characteristics.
For the classical problem: polynomial algorithms are known, a complete $H$-representation of a polytope with a polynomial number of inequalities is constructed,
$1$-skeleton of the polytope is completely described, and it is established that its clique number is polynomial in dimension.
At the same time, the problems with additional constraints are intractable, complete $H$-representations of the polytopes have not been found, $1$-skeletons of the polytopes are extremely complex: even the vertex adjacency test is an NP-complete problem, the clique numbers are superpolynomial in dimension.

\renewcommand{\refname}{References}
{
	}
\medskip
\end{document}